\documentclass[11pt,letterpaper,reqno]{amsart}
\usepackage{tikz}
\usetikzlibrary{positioning, shapes.geometric, arrows.meta}
\usepackage{amssymb}
\usepackage{amsmath}
\usepackage{amsthm}
\usepackage{amsfonts}
\usepackage{enumitem}
\usepackage{pgfplots}
\pgfplotsset{compat=1.18}
\usepackage{booktabs}

\usepackage{graphicx}
\usepackage[T1]{fontenc}
\usepackage{doi}
\usepackage{float}
\usepackage[dvipsnames]{xcolor}
\usepackage{hyperref}
\usepackage{bookmark}

\hypersetup{
	colorlinks=true,
	linkcolor=RoyalBlue,
	citecolor=ForestGreen!65!white,
	urlcolor=BrickRed,
	pdftitle={Rigidity of uniform Roe algebras over arbitrary uniformly locally finite coarse spaces},
	pdfauthor={Teng Zhang},
	pdfsubject={Rigidity of uniform Roe algebras and bijective coarse equivalence},
	pdfkeywords={uniform Roe algebra, coarse space, bijective coarse equivalence, rigidity},
	pdfstartview={FitH}
}

\newtheorem{thm}{Theorem}[section]
\newtheorem{lem}[thm]{Lemma}
\newtheorem{prop}[thm]{Proposition}
\newtheorem{cor}[thm]{Corollary}

\newtheorem{prob}[thm]{Problem}

\theoremstyle{definition}

\newtheorem{rem}[thm]{Remark}
\newtheorem{defin}[thm]{Definition}
\numberwithin{equation}{section}

\newcommand{\cE}{\mathcal E}
\newcommand{\cF}{\mathcal F}

\newcommand{\Cu}{C_u^*}
\newcommand{\supp}{\operatorname{supp}}
\newcommand{\Ad}{\operatorname{Ad}}
\newcommand{\id}{\operatorname{id}}

\newcommand{\diag}{\operatorname{diag}}

\begin{document}

\title[Isomorphism rigidity of uniform Roe algebras]
{Isomorphism rigidity of uniform Roe algebras over arbitrary uniformly locally finite coarse spaces}

%

\author[T.~Zhang]{Teng Zhang}

\address{School of Mathematics and Statistics, Xi'an Jiaotong University, Xi'an 710049, P. R. China}
\email{teng.zhang@stu.xjtu.edu.cn}



\thanks{Teng Zhang is supported by the China Scholarship Council, the Young Elite Scientists Sponsorship Program for PhD Students (China Association for Science and Technology), and the Fundamental Research Funds for the Central Universities at Xi'an Jiaotong University (Grant No.~xzy022024045).}

\begin{abstract}
	Let $(X,\cE)$ and $(Y,\cF)$ be uniformly locally finite coarse spaces. We prove that every $C^*$-algebra isomorphism
$
	\Cu(X,\cE)\cong \Cu(Y,\cF)
$
	forces $(X,\cE)$ and $(Y,\cF)$ to be bijectively coarsely equivalent. This completely resolves the isomorphism rigidity problem for uniform Roe algebras over arbitrary uniformly locally finite coarse spaces.
\end{abstract}

\maketitle

\section{Introduction}
 
 Uniform Roe algebras provide a natural operator algebraic model for
 large-scale geometry. They arose from the operator algebraic study of
 large-scale geometry initiated by Roe's work on index theory for
 noncompact manifolds. This circle of ideas is closely related to the
 coarse Baum--Connes conjecture; see Roe's monograph
 \cite[Chapters~2 and~4]{Roe03} and the groupoid formulation of
 Skandalis, Tu, and Yu \cite{STY02}. For a discrete coarse space, the
 uniform Roe algebra is the norm closure of the operators whose matrix
 supports are controlled by the coarse structure.
 The construction is functorial in the expected direction. A bijective
 coarse equivalence between two uniformly locally finite coarse spaces
 induces an isomorphism of their uniform Roe algebras, while a coarse
 equivalence induces a stable isomorphism. The rigidity problem asks to what extent these implications can be
 reversed. At the level of isomorphism, the natural question is whether
 the abstract uniform Roe algebra, without its distinguished canonical
 diagonal, determines the underlying coarse space up to bijective coarse
 equivalence. This leads to the following fundamental problem, for example, see \cite[Section~4]{BF21}, \cite[Problem~1.1]{BFV22} or \cite[Problem~4.8]{Vig26}.
 
 \begin{prob}[Isomorphism rigidity problem]\label{ques:bijective-rigidity}
 	Let $(X,\cE)$ and $(Y,\cF)$ be uniformly locally finite coarse spaces. If
$
 	\Cu(X,\cE)\cong \Cu(Y,\cF)
 $
 	as $C^*$-algebras, must $(X,\cE)$ and $(Y,\cF)$ be bijectively coarsely equivalent?
 \end{prob}

The first general rigidity results were obtained under property~A. For uniformly locally finite metric spaces, \v{S}pakula and Willett \cite[Theorem~4.1]{SW13} proved that isomorphisms of Roe-type algebras determine the coarse equivalence class of the underlying spaces. Roe and Willett \cite[Theorem~1.3]{RW14} later characterized property~A by the compactness of all ghost operators in the uniform Roe algebra, thereby identifying the main obstruction that subsequent rigidity arguments had to address. Sako  \cite[Theorem~4.1]{Sak14} further clarified the relevant localization principle by proving that property~A is equivalent to the operator norm localization property. White and Willett  \cite[Theorem~E and Corollary~6.13]{WW20} then developed an approach based on Cartan subalgebras and established strong uniqueness and rigidity results for uniform Roe algebras in the presence of property~A.

The passage from metric spaces to arbitrary coarse structures introduces several genuine difficulties. The coarse structure may fail to be countably generated, the associated Hilbert space may be nonseparable, and the space may have arbitrarily many coarse components. Braga and Farah~\cite{BF21} initiated the systematic study of rigidity for uniform Roe algebras in this general setting. 
Braga, Farah, and Vignati subsequently proved that if one of the spaces has property~A, then the spaces are bijectively coarsely equivalent \cite[Theorem~1.3]{BFV22}. They \cite[Theorem~1.4]{BFV22} also showed that if both uniform Roe algebras contain only compact ghost projections, then the spaces are coarsely equivalent.
In later work \cite{BFV24}, weaker forms of operator norm localization were developed for equi-approximable families of projections, leading to further rigidity results for nonmetrizable coarse spaces. These developments illustrate both the strength and the limitations of methods based on property~A, operator norm localization, and ghost operators.

For metric spaces, the rigidity problem without additional geometric assumptions was settled by Baudier, Braga, Farah, Khukhro, Vignati, and Willett \cite[Theorem~1.2]{BBFKVW22}. They \cite[Theorem~1.4]{BBFKVW22} also proved that, for uniformly locally finite metric spaces, coarse equivalence is equivalent to Morita equivalence of the corresponding uniform Roe algebras. Their argument obtains a uniform lower bound on suitable matrix coefficients by combining equi-approximability with a quantitative vector measure argument based ultimately on the Shapley--Folkman theorem. Mart\'inez and Vigolo \cite[Theorem~A]{MV25} later established unconditional rigidity for Roe algebras. Vignati \cite[Theorem~A]{Vig26} subsequently strengthened the metric uniform Roe algebra rigidity theorem by replacing coarse equivalence with bijective coarse equivalence. 
Krutoy \cite[Lemma~A and Theorem~B]{Kru26} subsequently related
bijective rigidity for admissible classes of metric spaces to the
injectivity of the zeroth comparison map for coarse groupoids and, using
Vignati's theorem, proved the unconditional injectivity of this map.

A substantial part of the countably generated case is also covered by the framework of coarse geometric modules developed by Mart\'inez and Vigolo \cite{MVframework,MVmodules}. In particular, their rigidity theorem shows that stable isomorphism of uniform Roe algebras implies coarse equivalence for coarsely locally finite, countably generated coarse spaces of the same coarse cardinality; see \cite[Corollary~8.3.5]{MVframework}. The general case in which the coarse structure is not assumed to be countably generated falls outside the scope of the countable diagonalization and Baire category arguments used in the metric theory.

The purpose of this paper is to resolve the isomorphism rigidity problem stated in Problem~\ref{ques:bijective-rigidity}.

\begin{thm}\label{thm:main}
    Let $(X,\cE)$ and $(Y,\cF)$ be uniformly locally finite coarse spaces. If
$
        \Cu(X,\cE)\cong \Cu(Y,\cF)
$
    as $C^*$-algebras, then $(X,\cE)$ and $(Y,\cF)$ are bijectively coarsely equivalent.
\end{thm}

We now describe the proof. First, the closed socle of a uniform Roe algebra is identified intrinsically as the $c_0$-direct sum of the compact-operator algebras associated with the coarse components. It follows that every isomorphism is implemented by a unitary
\[
    U:\ell^2(X)\longrightarrow \ell^2(Y)
\]
which permutes the coarse-component summands. Writing
\[
    u_{yx}=\langle U\delta_x,\delta_y\rangle,
\]
the first main step is the two-sided coefficient estimate
\begin{equation}\label{eq:intro-coefficients}
    \inf_{y\in Y}\sup_{x\in X}|u_{yx}|>0
    \quad\text{and}\quad
    \inf_{x\in X}\sup_{y\in Y}|u_{yx}|>0.
\end{equation}
The proof does not use property~A, ghost compactness, or a countable exhaustion of the coarse structure. Instead, it randomizes a diagonal unitary by independent Steinhaus variables. A fourth-moment estimate shows that, if a sequence of rows of $U$ were uniformly diffuse, then one random conjugate of a diagonal unitary would have columns whose mass cannot be captured by any fixed finite number of coordinates. This contradicts norm approximation by an entourage-supported operator.

The estimates in \eqref{eq:intro-coefficients} provide uniformly finite-to-one maps $f:X\to Y$ and $g:Y\to X$ with uniformly large selected coefficients. A direct conjugation of a partial translation, however, contains cross terms that may cancel the selected main coefficient. The second main step is therefore a finite double $\ell^1$-paving lemma. It partitions each relevant matching into finitely many classes on which the total cross term is smaller than half of the main term. A fixed positive coefficient threshold then detects every class as an entourage. This proves that $f$ and $g$ are controlled. Applying the same paving argument to diagonal projections shows that $g\circ f$ and $f\circ g$ are close to the corresponding identity maps. To obtain a bijection, we finitely partition $X$ so that the compressions of $U$ along $f$ are invertible. Their polar parts give an explicit Murray--von Neumann equivalence in a finite matrix amplification of $\Cu(Y,\cF)$. Approximating this equivalence by an entourage-supported invertible operator produces a locally finite bipartite graph satisfying Hall's condition on both sides, and hence a bijection close to $f$.

The paper is organized as follows. Section~\ref{sec:preliminaries} fixes notation and records elementary support and coloring facts. Section~\ref{sec:spatial} proves spatial implementation in the presence of arbitrary coarse components. Section~\ref{sec:random} establishes the random diagonal-unitary lemma and derives \eqref{eq:intro-coefficients}. Section~\ref{sec:paving} proves the finite double paving result. Section~\ref{sec:proof} constructs controlled coarse inverses, and Section~\ref{sec:bijective} upgrades them to a bijective coarse equivalence.

\noindent
\section{Coarse spaces and uniform Roe algebras}\label{sec:preliminaries}

We begin with the definitions needed throughout the proof. Our conventions agree with those used for general coarse structures in \cite[Section~2, pp.~305--313]{BFV22}.

\begin{defin}
	Let $X$ be a set. A \emph{coarse structure} on $X$ is a family $\cE$ of subsets of $X\times X$ containing
	\[
		\Delta_X=\{(x,x):x\in X\}
	\]
	and closed under taking subsets, finite unions, inverses, and compositions. The elements of $\cE$ are called \emph{entourages}.

	For $E\subseteq X\times X$ and $x\in X$, set
	\[
		E[x]=\{y\in X:(x,y)\in E\}.
	\]
	The coarse space $(X,\cE)$ is \emph{uniformly locally finite} if
	\begin{equation}\label{eq:ulf}
		N_E:=\sup_{x\in X}|E[x]|<\infty
	\end{equation}
	for every $E\in\cE$.
\end{defin}

Since $E^{-1}$ is also an entourage, uniform local finiteness bounds both the horizontal and vertical sections of every entourage.

Two maps $h_1,h_2:X\to Y$ are \emph{close}, written $h_1\sim h_2$, if
\[
    \{(h_1(x),h_2(x)):x\in X\}\in\cF.
\]
Closeness is an equivalence relation. Moreover, if $h_1\sim h_2$ and
$k:Y\to Z$ is controlled, then
$
k\circ h_1\sim k\circ h_2.
$

A map $f:X\to Y$ is \emph{controlled} if $(f\times f)(E)\in\cF$ for all $E\in\cE$. The spaces $(X,\cE)$ and $(Y,\cF)$ are \emph{coarsely equivalent} if there are controlled maps $f:X\to Y$ and $g:Y\to X$ such that $g\circ f\sim\id_X$ and $f\circ g\sim\id_Y$. They are \emph{bijectively coarsely equivalent} if there is a bijection $b:X\to Y$ such that both $b$ and $b^{-1}$ are controlled. These conditions also imply that $f$ and $g$ are proper. Indeed, if $B\subseteq Y$ is bounded, so that $B\times B\in\cF$, then $g(B)\times g(B)\in\cE$. Composing this entourage with the entourage witnessing $g\circ f\sim\id_X$ and its inverse shows that $f^{-1}(B)\times f^{-1}(B)\in\cE$. The argument for $g$ is symmetric. Thus this agrees with the usual notion of coarse equivalence.

For $x,x'\in X$, let $e_{xx'}\in B(\ell^2(X))$ denote the matrix unit determined by
\[
	e_{xx'}\delta_z=
	\begin{cases}
		\delta_x,&z=x',\\
		0,&z\ne x'.
	\end{cases}
\]
For $T\in B(\ell^2(X))$, its support is
\[
	\supp(T)=\{(x,x'):\langle T\delta_{x'},\delta_x\rangle\ne0\}.
\]
The uniform Roe algebra is
\[
	\Cu(X,\cE)
	=
	\overline{\{T\in B(\ell^2(X)):\supp(T)\in\cE\}}^{\|\cdot\|}.
\]
In particular,
\begin{equation}\label{eq:diagonal-inclusion}
	\ell^\infty(X)\subseteq \Cu(X,\cE),
\end{equation}
because every diagonal operator is supported on $\Delta_X$.

The following elementary observation will be used repeatedly to convert a uniform lower bound on matrix entries into an entourage.

\begin{lem}\label{lem:threshold}
	Let $(X,\cE)$ be a uniformly locally finite coarse space, let $T\in\Cu(X,\cE)$, and let $c>0$. Then
	\[
		\supp_c(T)
		:=
		\{(x,x')\in X\times X:
		|\langle T\delta_{x'},\delta_x\rangle|\ge c\}
	\]
	is an entourage.
\end{lem}

\begin{proof}
	Choose an entourage $E\in\cE$ and an $E$-supported operator $S$ such that
	\[
		\|T-S\|<c.
	\]
	If $(x,x')\notin E$, then
	\[
		|\langle T\delta_{x'},\delta_x\rangle|
		=
		|\langle (T-S)\delta_{x'},\delta_x\rangle|
		<c.
	\]
	Consequently, $\supp_c(T)\subseteq E$, and closure of $\cE$ under subsets completes the proof.
\end{proof}

We next record a decomposition of bounded-degree relations. A relation $R\subseteq X\times Y$ is called a \emph{matching} if its first-coordinate projection and its second-coordinate projection are both injective on $R$.

\begin{lem}\label{lem:matching-decomposition}
	Let $R\subseteq X\times Y$ satisfy
	\[
		\sup_{x\in X}|\{y:(x,y)\in R\}|\le d
	\]
	and
	\[
		\sup_{y\in Y}|\{x:(x,y)\in R\}|\le d
	\]
	for some $d\in\mathbb N$. Then $R$ is the union of at most $d$ matchings.
\end{lem}

\begin{proof}
	Regard $R$ as the edge set of the bipartite graph with vertex classes $X$ and $Y$. Every finite subgraph has maximum degree at most $d$, so
	K\H{o}nig's line-colouring theorem
	\cite[p.~453, Theorem~17.2]{BM08}
	gives a proper edge colouring with $d$ colors. Apply the finite intersection property in the compact product space $\{1,\ldots,d\}^{R}$ to obtain a proper $d$-edge-coloring of the entire graph. Each color class is a matching.
\end{proof}

Thus every entourage in a uniformly locally finite coarse space is a finite union of graphs of partial bijections. We shall use Lemma~\ref{lem:matching-decomposition} not only for entourages, but also for their images under uniformly finite-to-one maps.

\section{Coarse components and spatial implementation}\label{sec:spatial}

An arbitrary coarse space need not be coarsely connected, and $\ell^2(X)$ need not be separable. The first task is therefore to isolate the intrinsic ideal which forces an algebra isomorphism to be spatial.

Define an equivalence relation on $X$ by
\begin{equation}\label{eq:component-relation}
	x\sim_{\cE}x'
	\quad\Longleftrightarrow\quad
	(x,x')\in E\ \text{for some }E\in\cE.
\end{equation}
Closure of $\cE$ under inverses and compositions shows that \eqref{eq:component-relation} is indeed an equivalence relation. The set of equivalence classes is denoted by $\pi_0(X,\cE)$, and its elements are called the \emph{coarse components}.

Every entourage is contained in the union of the squares of the coarse components. It follows that
\begin{equation}\label{eq:block-decomposition}
	\ell^2(X)=\bigoplus_{C\in\pi_0(X,\cE)}\ell^2(C)
\end{equation}
reduces every operator in $\Cu(X,\cE)$.

\begin{prop}\label{prop:socle}
	Let $(X,\cE)$ be a uniformly locally finite coarse space. The minimal projections of $\Cu(X,\cE)$ are precisely the Hilbert-space rank-one projections supported on one coarse component. Consequently, the closed linear span of the minimal projections is
	\begin{equation}\label{eq:socle}
		J_X
		=
		\bigoplus_{C\in\pi_0(X,\cE)}^{c_0}
		K(\ell^2(C)).
	\end{equation}
	Moreover, $J_X$ is an essential ideal of $\Cu(X,\cE)$.
\end{prop}

\begin{proof}
	Fix a coarse component $C$. If $F\subseteq C$ is finite, then $F\times F$ is an entourage: for each pair in $F\times F$, choose an entourage containing it and take the finite union. Hence every operator whose matrix is supported on a finite subset of $C\times C$ belongs to $\Cu(X,\cE)$. Since such operators are norm dense in $K(\ell^2(C))$, we obtain
	\[
		K(\ell^2(C))\subseteq \Cu(X,\cE).
	\]

	A Hilbert-space rank-one projection in $K(\ell^2(C))$ is plainly minimal in $\Cu(X,\cE)$. Conversely, let $p$ be a minimal projection in $\Cu(X,\cE)$. By \eqref{eq:block-decomposition}, $p$ is block diagonal. For every coarse component $C$, the projection $\chi_C$ belongs to $\ell^\infty(X)\subseteq\Cu(X,\cE)$. If two component blocks of $p$ were nonzero, then $p\chi_C$ would be a nonzero proper subprojection of $p$ for one of those components. Thus $p$ is supported on one component $C$. If $p$ had Hilbert-space rank greater than one, it would dominate a nonzero proper rank-one projection belonging to $K(\ell^2(C))$, again contradicting minimality.

	The complex linear span of the rank-one projections on $\ell^2(C)$ contains all finite-rank operators: by polarization, every rank-one operator is a linear combination of rank-one positive operators, and each nonzero rank-one positive operator is a scalar multiple of a rank-one projection. It follows that the closed linear span of all minimal projections is the norm closure of the algebraic direct sum of the $K(\ell^2(C))$, which is precisely the $c_0$-direct sum in \eqref{eq:socle}.

	To verify that this $c_0$-sum is an ideal, write $a=\bigoplus_C a_C\in\Cu(X,\cE)$ and $k=(k_C)_C\in J_X$ using \eqref{eq:block-decomposition}. Each $a_Ck_C$ and $k_Ca_C$ is compact, and
\[
\max\{\|a_Ck_C\|,\|k_Ca_C\|\}
\le \|a\|\,\|k_C\|.
\]
	Since $\|k_C\|\to0$ in the $c_0$-sense, both $ak$ and $ka$ belong to $J_X$. Finally, if $aJ_X=0$, then $ae_{xx}=0$ for every $x\in X$, and hence $a\delta_x=0$ for every $x$. Thus $a=0$, so $J_X$ is essential.
\end{proof}

The description in Proposition~\ref{prop:socle} is purely $C^*$-algebraic: $J_X$ is the closed socle. It is therefore preserved by every isomorphism.

\begin{prop}\label{prop:spatial}
	Let $(X,\cE)$ and $(Y,\cF)$ be uniformly locally finite coarse spaces. Every $C^*$-algebra isomorphism
	\[
		\Phi:\Cu(X,\cE)\longrightarrow\Cu(Y,\cF)
	\]
	is spatially implemented. More precisely, there is a unitary
	\[
		U:\ell^2(X)\longrightarrow\ell^2(Y)
	\]
	which sends each coarse-component summand of \eqref{eq:block-decomposition} onto a coarse-component summand for $Y$ and satisfies
	\begin{equation}\label{eq:spatial-implementation}
		\Phi(a)=UaU^*,
		\qquad a\in\Cu(X,\cE).
	\end{equation}
\end{prop}

\begin{proof}
	By Proposition~\ref{prop:socle},
	\[
		\Phi(J_X)=J_Y.
	\]
	The summands $K(\ell^2(C))$, where $C\in\pi_0(X,\cE)$, are precisely the minimal nonzero closed ideals of $J_X$: each summand is simple, and distinct summands annihilate one another. Hence there is a bijection
	\[
		\sigma:\pi_0(X,\cE)\longrightarrow\pi_0(Y,\cF)
	\]
	such that $\Phi$ restricts to an isomorphism
	\[
		K(\ell^2(C))\longrightarrow K(\ell^2(\sigma(C)))
	\]
	for every $C$.

	We recall why such an isomorphism is spatial without any separability assumption. Choose an orthonormal basis $(\xi_i)_{i\in I}$ of $\ell^2(C)$, with associated matrix units $(e_{ij})_{i,j\in I}$, and fix $i_0\in I$. If $\eta_{i_0}$ is a unit vector in the range of $\Phi(e_{i_0i_0})$, set
	\[
		\eta_i=\Phi(e_{ii_0})\eta_{i_0},
		\qquad i\in I.
	\]
	The family $(\eta_i)_{i\in I}$ is orthonormal and $\Phi(e_{ij})\eta_k=\delta_{jk}\eta_i$. Its closed span is all of $\ell^2(\sigma(C))$: otherwise the nonzero orthogonal complement would be annihilated by the range algebra $K(\ell^2(\sigma(C)))$. Thus the unitary $U_C\xi_i=\eta_i$ implements the restriction of $\Phi$ to $K(\ell^2(C))$. Taking the Hilbert-space direct sum of the $U_C$ gives a unitary $U:\ell^2(X)\to\ell^2(Y)$ such that
	\begin{equation}\label{eq:spatial-on-socle}
		\Phi(k)=UkU^*,
		\qquad k\in J_X.
	\end{equation}

	It remains to pass from $J_X$ to the whole uniform Roe algebra. Let $a\in\Cu(X,\cE)$ and $k\in J_X$. By \eqref{eq:spatial-on-socle},
	\[
	\begin{aligned}
		\Phi(a)\Phi(k)
		&=\Phi(ak)
		=UakU^*\\
		&=(UaU^*)\Phi(k).
	\end{aligned}
	\]
	Thus
	\[
		(\Phi(a)-UaU^*)b=0,
		\qquad b\in J_Y.
	\]
	This equality takes place in $B(\ell^2(Y))$ and does not presuppose that $UaU^*$ belongs to the target algebra. Taking $b=e_{yy}$ for every $y\in Y$ shows that the difference annihilates every canonical basis vector. Therefore \eqref{eq:spatial-implementation} holds.
\end{proof}

In what follows, we fix an isomorphism $\Phi$ and a unitary $U$ as in Proposition~\ref{prop:spatial}, and write
\begin{equation}\label{eq:unitary-coefficients}
	u_{yx}=\langle U\delta_x,\delta_y\rangle,
	\qquad x\in X,\ y\in Y.
\end{equation}

\section{A random diagonal-unitary lemma}\label{sec:random}

The key analytic step is to obtain a uniform lower bound for the largest
coefficient in each row of a unitary that conjugates the full atomic
diagonal algebra into a uniform Roe algebra. This estimate does not
require the coarse space to be metrizable. Its proof uses a random
diagonal unitary and repeatedly interchanges nonnegative sums with
integrals or expectations. We therefore need the following standard
form of Tonelli's theorem for nonnegative series; see
\cite[p.~67, Theorem~2.37]{Fol99}.

\begin{lem}[\cite{Fol99}]\label{lem:tonelli}
	Let $(\Omega,\Sigma,\mu)$ be a measure space, and let
	$(f_n)_{n\in\mathbb N}$ be a sequence of nonnegative measurable
	functions on $\Omega$. Then
	\[
	\int_\Omega \sum_{n=1}^{\infty} f_n\,d\mu
	=
	\sum_{n=1}^{\infty}\int_\Omega f_n\,d\mu,
	\]
	where both sides are allowed to take the value $+\infty$.
	
	In particular, if $I$ and $J$ are countable sets and
	$(a_{ij})_{(i,j)\in I\times J}$ is a family of nonnegative real
	numbers, then
	\[
	\sum_{i\in I}\sum_{j\in J}a_{ij}
	=
	\sum_{j\in J}\sum_{i\in I}a_{ij}.
	\]
	Moreover, if $(X_n)_{n\in\mathbb N}$ is a sequence of nonnegative
	random variables, then
	\[
	\mathbb E\left[\sum_{n=1}^{\infty}X_n\right]
	=
	\sum_{n=1}^{\infty}\mathbb E[X_n].
	\]
\end{lem}

We now establish the required coefficient estimate.

\begin{lem}\label{lem:random-diagonal}
	Let $(Y,\cF)$ be a nonempty uniformly locally finite coarse space, let $X$ be a nonempty set, and let
$
		U:\ell^2(X)\longrightarrow\ell^2(Y)
$
	be unitary. If
	\begin{equation}\label{eq:diagonal-conjugacy}
		U\ell^\infty(X)U^*
		\subseteq
		\Cu(Y,\cF),
	\end{equation}
	then
	\begin{equation}\label{eq:row-lower-bound}
		\inf_{y\in Y}\sup_{x\in X}
		|\langle U\delta_x,\delta_y\rangle|>0.
	\end{equation}
\end{lem}

\begin{proof}
	Write $u_{yx}=\langle U\delta_x,\delta_y\rangle$. Suppose that \eqref{eq:row-lower-bound} fails, and put
    \[
        m(y):=\sup_{x\in X}|u_{yx}|^2,
        \qquad y\in Y.
    \]
    Then $\inf_{y\in Y}m(y)=0$, while $m(y)>0$ for every $y$, since $U^*\delta_y$ is a nonzero unit vector. Choose $y_1$ with $m(y_1)<2^{-1}$ and, recursively, choose $y_n$ so that
    \[
        m(y_n)<\min\{2^{-n},m(y_1),\ldots,m(y_{n-1})\}.
    \]
    The points $y_n$ are pairwise distinct, and
    \begin{equation}\label{eq:mn-definition}
        m_n:=\sup_{x\in X}|u_{y_nx}|^2<2^{-n},
        \qquad n\in\mathbb N.
    \end{equation}
    In particular, $\sum_n m_n<\infty$.

	We first remove a minor measurability issue caused by arbitrary cardinalities. For each $n$, the vector $U^*\delta_{y_n}\in\ell^2(X)$ has countable support. Hence
	\[
		S=\bigcup_{n=1}^{\infty}\supp(U^*\delta_{y_n})
	\]
	is countable. Moreover,
	\[
		Z=\bigcup_{x\in S}\supp(U\delta_x)
	\]
	is a countable subset of $Y$. Only the coordinates in $S$ will be randomized.

	On the product probability space $\mathbb T^S$, endowed with product Haar probability measure, let $(\varepsilon_x)_{x\in S}$ be the coordinate Steinhaus random variables. Put $\varepsilon_x=1$ for $x\notin S$, and define the diagonal unitary
	\[
		D_\varepsilon\delta_x=\varepsilon_x\delta_x.
	\]
	By \eqref{eq:diagonal-conjugacy}, every realization of
    \[
        T=UD_\varepsilon U^*
    \]
	belongs to $\Cu(Y,\cF)$. Writing $T_{z,y}=\langle T\delta_y,\delta_z\rangle$, for $z\in Y$ and $n\in\mathbb N$ we have
	\begin{equation}\label{eq:random-entry}
		T_{z,y_n}
		=
		\sum_{x\in S}
		\varepsilon_xu_{zx}\overline{u_{y_nx}}.
	\end{equation}
	For each fixed $z$ and $n$, the series in \eqref{eq:random-entry} is absolutely convergent, since
    \[
        \sum_{x\in S}|u_{zx}|\,|u_{y_nx}|
        \le
        \left(\sum_x|u_{zx}|^2\right)^{1/2}
        \left(\sum_x|u_{y_nx}|^2\right)^{1/2}
        =1.
    \]
    Moreover, $T\delta_{y_n}$ is supported in the fixed countable set $Z$.

	Set
	\[
		v_{z,n}
		=
		\sum_{x\in S}|u_{zx}|^2|u_{y_nx}|^2.
	\]
	For completeness, if $(\zeta_k)_{k=1}^N$ are independent Steinhaus variables and $a_1,\ldots,a_N\in\mathbb C$, then
	\[
		\mathbb E\left|\sum_{k=1}^N a_k\zeta_k\right|^4
		=
		2\left(\sum_{k=1}^N|a_k|^2\right)^2
		-
		\sum_{k=1}^N|a_k|^4.
	\]
	Indeed, after expansion, a mixed fourth-order term has nonzero expectation exactly when the multisets of unbarred and barred indices coincide. The finite subsums in \eqref{eq:random-entry} converge uniformly in the Steinhaus variables, because the series is absolutely summable. They therefore converge in $L^4(\mathbb T^S)$. Applying this identity to the finite subsums and passing to the limit gives
	\begin{equation}\label{eq:fourth-moment-pointwise}
		\mathbb E|T_{z,y_n}|^4
		=
		2v_{z,n}^2
		-
		\sum_{x\in S}|u_{zx}|^4|u_{y_nx}|^4
		\le 2v_{z,n}^2.
	\end{equation}
	Since $\supp(U^*\delta_{y_n})\subseteq S$ and $\supp(U\delta_x)\subseteq Z$ for every $x\in S$, Lemma~\ref{lem:tonelli} and unitarity give
	\begin{equation}\label{eq:v-sum}
	\begin{aligned}
		\sum_{z\in Z}v_{z,n}
		&=
		\sum_{x\in S}|u_{y_nx}|^2
		\sum_{z\in Z}|u_{zx}|^2
		=1,\\
		v_{z,n}
		&\le
		m_n\sum_{x\in S}|u_{zx}|^2
		\le m_n.
	\end{aligned}
	\end{equation}
	Combining \eqref{eq:fourth-moment-pointwise} and \eqref{eq:v-sum}, and using Lemma~\ref{lem:tonelli} once more, we obtain
	\begin{equation}\label{eq:fourth-moment-global}
		\mathbb E\sum_{z\in Z}|T_{z,y_n}|^4
		\le
		2\sum_{z\in Z}v_{z,n}^2
		\le 2m_n.
	\end{equation}
	Define
    \[
        q_n:=\sum_{z\in Z}|T_{z,y_n}|^4.
    \] Each $q_n$ is measurable, since $Z$ is countable and every matrix
    coefficient in \eqref{eq:random-entry} is measurable.
    By \eqref{eq:mn-definition}, \eqref{eq:fourth-moment-global}, and Lemma~\ref{lem:tonelli},
    \[
        \mathbb E\left(\sum_{n=1}^{\infty}q_n\right)
        \le 2\sum_{n=1}^{\infty}m_n<\infty.
    \]
    Hence $\sum_n q_n<\infty$ almost surely. Fix a realization with this property. Then
    \begin{equation}\label{eq:q-to-zero}
        q_n\longrightarrow0.
    \end{equation}

	For a unit vector $\xi\in\ell^2(Y)$ and $N\in\mathbb N$, define
	\[
		L_N(\xi)
		=
		\sup\{\|\chi_A\xi\|^2:A\subseteq Y,\ |A|\le N\}.
	\]
	For every $A\subseteq Y$ with $|A|\le N$, the Cauchy--Schwarz inequality gives
	\[
		\left(\sum_{z\in A}|T_{z,y_n}|^2\right)^2
		\le |A|\sum_{z\in A}|T_{z,y_n}|^4
		\le Nq_n.
	\]
	Taking the supremum over $A$ yields
	\begin{equation}\label{eq:top-N-bound}
		L_N(T\delta_{y_n})^2
		\le Nq_n.
	\end{equation}
	Thus, for every fixed $N$, equations \eqref{eq:q-to-zero} and \eqref{eq:top-N-bound} imply
	\begin{equation}\label{eq:top-N-vanishing}
		L_N(T\delta_{y_n})\longrightarrow0.
	\end{equation}

	On the other hand, $T\in\Cu(Y,\cF)$. Choose $F\in\cF$ and an $F$-supported operator $S_0$ such that
	\[
		\|T-S_0\|<\frac14.
	\]
	Since $F^{-1}\in\cF$, uniform local finiteness gives
    \[
        N_{F^{-1}}:=\sup_{y\in Y}|F^{-1}[y]|<\infty.
    \]
    Put $A_y=F^{-1}[y]$. The vector $S_0\delta_y$ is supported on $A_y$, and therefore
    \[
        \|\chi_{Y\setminus A_y}T\delta_y\|
        =\|\chi_{Y\setminus A_y}(T-S_0)\delta_y\|
        <\frac14.
    \]
    Since $T$ is unitary, it follows that
    \[
        L_{N_{F^{-1}}}(T\delta_y)
        \ge \|\chi_{A_y}T\delta_y\|^2
        >1-\frac1{16}=\frac{15}{16}
    \]
    for every $y\in Y$. Taking $y=y_n$ contradicts \eqref{eq:top-N-vanishing}. This proves \eqref{eq:row-lower-bound}.
\end{proof}

Applying Lemma~\ref{lem:random-diagonal} to an isomorphism and then to its inverse produces lower bounds in both directions.

\begin{cor}\label{cor:two-sided-coefficients}
	Let $(X,\cE)$ and $(Y,\cF)$ be nonempty uniformly locally finite coarse spaces, let $\Phi:\Cu(X,\cE)\to\Cu(Y,\cF)$ be an isomorphism, and let $U$ implement $\Phi$ as in Proposition~\ref{prop:spatial}. Then
	\[
		\inf_{y\in Y}\sup_{x\in X}|u_{yx}|>0,
		\qquad
		\inf_{x\in X}\sup_{y\in Y}|u_{yx}|>0.
	\]
\end{cor}

\begin{proof}
	The inclusion
	\[
		U\ell^\infty(X)U^*
		\subseteq \Cu(Y,\cF)
	\]
	follows from \eqref{eq:diagonal-inclusion} and \eqref{eq:spatial-implementation}, so Lemma~\ref{lem:random-diagonal} gives the first inequality. Applying the lemma to $U^*$ and
	\[
		U^*\ell^\infty(Y)U
		\subseteq \Cu(X,\cE)
	\]
	gives the second.
\end{proof}

Choose constants $\alpha,\beta>0$ such that
\[
    0<\alpha<\inf_{x\in X}\sup_{y\in Y}|u_{yx}|,
    \qquad
    0<\beta<\inf_{y\in Y}\sup_{x\in X}|u_{yx}|.
\]
We may therefore select maps
\begin{equation}\label{eq:f-g-selection}
	f:X\longrightarrow Y,
	\qquad
	g:Y\longrightarrow X
\end{equation}
such that
\begin{equation}\label{eq:selected-coefficients}
	|u_{f(x),x}|\ge\alpha,
	\qquad
	|u_{y,g(y)}|\ge\beta.
\end{equation}
These maps are uniformly finite-to-one. Indeed, orthonormality of the rows and columns of $U$ gives
\begin{equation}\label{eq:fibre-bounds}
	|f^{-1}(y)|\le\alpha^{-2},
	\qquad
	|g^{-1}(x)|\le\beta^{-2}.
\end{equation}

The remaining question is whether the maps in \eqref{eq:f-g-selection} are controlled and close to mutual inverses. The next section develops the finite paving needed to overcome cancellation in conjugated partial translations.

\section{A finite double paving lemma}\label{sec:paving}

The following lemma is a weighted finite-coloring principle. Sums of nonnegative families over arbitrary index sets are understood as suprema of finite subsums. The simultaneous assumptions on row sums and column sums are essential for the symmetric reduction used in the proof.

\begin{lem}\label{lem:double-paving}
	Let $I$ be an arbitrary set and let $(b_{ij})_{i,j\in I}$ be a family of nonnegative real numbers such that
	\begin{equation}\label{eq:paving-hypotheses}
		b_{ii}=0,
		\qquad
		\sup_{i\in I}\sum_{j\in I}b_{ij}\le M,
		\qquad
		\sup_{j\in I}\sum_{i\in I}b_{ij}\le M.
	\end{equation}
	Then, for every $\eta>0$, there is a finite partition
$
        I=I_1\sqcup\cdots\sqcup I_r
$
	such that
	\begin{equation}\label{eq:paving-conclusion}
		\sum_{\substack{j\in I_k\\j\ne i}}b_{ij}<\eta,
		\qquad i\in I_k,
	\end{equation}
	for every $k\in\{1,\ldots,r\}$.
\end{lem}

\begin{proof}
	Define the symmetric weights
$
		w_{ij}=b_{ij}+b_{ji}.
$
	By \eqref{eq:paving-hypotheses},
	\begin{equation}\label{eq:symmetric-weight-bound}
		\sum_{j\in I}w_{ij}\le2M,
		\qquad i\in I.
	\end{equation}
	Choose $r\in\mathbb N$ such that
	\begin{equation}\label{eq:number-colours}
		\frac{2M}{r}<\eta.
	\end{equation}

	Suppose first that $I$ is finite. Among all maps $c:I\to\{1,\ldots,r\}$, choose one minimizing the total weight of monochromatic unordered pairs. For $i\in I$ and $\ell\in\{1,\ldots,r\}$, set
	\[
		W_\ell(i)=\sum_{\substack{j\ne i\\c(j)=\ell}}w_{ij}.
	\]
	If $c(i)=k$, minimality under recoloring the single vertex $i$ implies
	\[
		W_k(i)\le W_\ell(i)
		\qquad(1\le\ell\le r).
	\]
	Consequently, \eqref{eq:symmetric-weight-bound} gives
	\[
		rW_k(i)
		\le
		\sum_{\ell=1}^{r}W_\ell(i)
		\le2M.
	\]
	Together with \eqref{eq:number-colours} and the inequality $b_{ij}\le w_{ij}$, this proves \eqref{eq:paving-conclusion} in the finite case.

	For general $I$, consider the compact product space
	\[
		\Omega=\{1,\ldots,r\}^{I}.
	\]
	For $i\in I$ and finite $J\subseteq I\setminus\{i\}$, impose the closed condition
	\[
		\sum_{\substack{j\in J\\c(j)=c(i)}}w_{ij}
		\le \frac{2M}{r}.
	\]
	Every finite family of these conditions involves a finite set of vertices. Applying the finite argument to the weights restricted to that set gives a coloring satisfying the chosen conditions. The finite intersection property therefore gives a coloring satisfying all of them. Taking the supremum over finite $J$, using \eqref{eq:number-colours}, and again noting that $b_{ij}\le w_{ij}$ yields \eqref{eq:paving-conclusion}.
\end{proof}

The role of Lemma~\ref{lem:double-paving} is now transparent. A selected coefficient of $Ue_{ab}U^*$ contributes a main term of fixed size, while all other selected matrix units contribute a matrix whose row and column $\ell^1$ norms are bounded by the Cauchy--Schwarz inequality. The finite paving makes the total unwanted contribution strictly smaller than the main term.

\section{Construction of the coarse equivalence}\label{sec:proof}

We retain the notation of Sections~\ref{sec:spatial} and~\ref{sec:random}. Thus $\Phi=\Ad(U)$, the coefficients of $U$ are given by \eqref{eq:unitary-coefficients}, and the maps $f$ and $g$ satisfy \eqref{eq:selected-coefficients} and \eqref{eq:fibre-bounds}.

We first prove controlledness.

\begin{prop}\label{prop:f-controlled}
	The map $f:X\to Y$ selected in \eqref{eq:f-g-selection} is controlled.
\end{prop}

\begin{proof}
	Fix $E\in\cE$. If $E=\varnothing$, there is nothing to prove, so assume that $E\ne\varnothing$. Set
	\begin{equation}\label{eq:image-relation}
		R=(f\times f)(E)\subseteq Y\times Y.
	\end{equation}
	Let
$
		K_f=\sup_{y\in Y}|f^{-1}(y)|<\infty.
$
	By \eqref{eq:ulf} and \eqref{eq:fibre-bounds}, the horizontal sections of $R$ have cardinality at most $K_fN_E$, while the vertical sections have cardinality at most $K_fN_{E^{-1}}$. With
$
        d=K_f\max\{N_E,N_{E^{-1}}\},
$
    Lemma~\ref{lem:matching-decomposition} therefore writes $R$ as a union of at most $d$ matchings. It suffices to show that each such matching is an entourage.

	Fix one matching
	\begin{equation}\label{eq:target-matching}
		R_0=\{(y_i,z_i):i\in I\}\subseteq R.
	\end{equation}
	For every $i\in I$, choose $(a_i,b_i)\in E$ such that
	\begin{equation}\label{eq:chosen-preimages}
		f(a_i)=y_i,
		\qquad
		f(b_i)=z_i.
	\end{equation}
	Because the first coordinates $y_i$ in \eqref{eq:target-matching} are distinct, the points $a_i$ are distinct; similarly, the $b_i$ are distinct.

	For $i,j\in I$, put
$
		c_{ij}=u_{y_i,a_j}\overline{u_{z_i,b_j}}
$
	and define
	\[
		d_{ij}=
		\begin{cases}
			|c_{ij}|,&i\ne j,\\
			0,&i=j.
		\end{cases}
	\]
	Since the families $(a_j)_j$ and $(b_j)_j$ are injective, the Cauchy--Schwarz inequality gives
    \[
        \sum_{j\in I}d_{ij}
        \le
        \left(\sum_j|u_{y_i,a_j}|^2\right)^{1/2}
        \left(\sum_j|u_{z_i,b_j}|^2\right)^{1/2}
        \le1.
    \]
    The families $(y_i)_i$ and $(z_i)_i$ are also injective, so another application of the Cauchy--Schwarz inequality gives
    \[
        \sum_{i\in I}d_{ij}\le1.
    \]
	Apply Lemma~\ref{lem:double-paving}, with $M=1$ and $\eta=\alpha^2/2$, to obtain a finite partition
$
		I=I_1\sqcup\cdots\sqcup I_r
$
	such that
	\begin{equation}\label{eq:cross-small}
		\sum_{\substack{j\in I_k\\j\ne i}}d_{ij}
		<\frac{\alpha^2}{2},
		\qquad i\in I_k.
	\end{equation}

	For each $k$, define a partial isometry $V_k\in B(\ell^2(X))$ by
	\[
		V_k\delta_{b_j}=\delta_{a_j}
		\quad(j\in I_k),
	\]
	and let $V_k$ vanish on the orthogonal complement of the closed linear span of the $\delta_{b_j}$. The two endpoint families are injective, so $V_k$ is well defined and
	\[
		\supp(V_k)=\{(a_j,b_j):j\in I_k\}\subseteq E.
	\]
	Hence $V_k\in\Cu(X,\cE)$. The coefficient series below is absolutely convergent by the same Cauchy--Schwarz estimate used in the row bound. For $i\in I_k$, equations \eqref{eq:selected-coefficients}, \eqref{eq:chosen-preimages}, and \eqref{eq:cross-small} imply
	\begin{align}
	\left|
	\langle UV_kU^*\delta_{z_i},\delta_{y_i}\rangle
	\right|
	&=
	\left|
	\sum_{j\in I_k}u_{y_i,a_j}\overline{u_{z_i,b_j}}
	\right|\notag\\
	&\ge
	|u_{y_i,a_i}|\,|u_{z_i,b_i}|
	-
	\sum_{\substack{j\in I_k\\j\ne i}}d_{ij}\notag\\
	&>\alpha^2-\frac{\alpha^2}{2}
	=\frac{\alpha^2}{2}.
	\label{eq:controlled-threshold}
	\end{align}
	Since $UV_kU^*=\Phi(V_k)\in\Cu(Y,\cF)$, Lemma~\ref{lem:threshold} and \eqref{eq:controlled-threshold} show that
	\[
		\{(y_i,z_i):i\in I_k\}\in\cF.
	\]
	There are only finitely many paving classes and finitely many matchings, so finite union closure gives $R\in\cF$. In view of \eqref{eq:image-relation}, $f$ is controlled.
\end{proof}

Applying Proposition~\ref{prop:f-controlled} to $\Phi^{-1}=\Ad(U^*)$ yields the symmetric conclusion.

\begin{cor}\label{cor:g-controlled}
	The map $g:Y\to X$ selected in \eqref{eq:f-g-selection} is controlled.
\end{cor}

It remains to establish closeness of the two composites to the identity maps. The argument again uses the double paving, now applied to diagonal projections rather than partial translations.

\begin{prop}\label{prop:close-X}
	The map $g\circ f$ is close to $\id_X$.
\end{prop}

\begin{proof}
	Consider the relation
    \[
        Q_X=\{(x,g(f(x))):x\in X\}.
    \]
    Its horizontal sections have cardinality one. For $b\in X$,
    \[
        |\{x:g(f(x))=b\}|
        \le \sum_{y\in g^{-1}(b)}|f^{-1}(y)|
        \le \alpha^{-2}\beta^{-2}.
    \]
    Thus its vertical sections are uniformly bounded. Hence Lemma~\ref{lem:matching-decomposition} writes $Q_X$ as a finite union of matchings.

	Fix one matching
	\begin{equation}\label{eq:close-matching-X}
		\{(a_i,b_i):i\in I\}\subseteq Q_X
	\end{equation}
	and put
	\begin{equation}\label{eq:yi-close-X}
		y_i=f(a_i).
	\end{equation}
	Then $b_i=g(y_i)$. The family $(y_i)_i$ is injective: if $y_i=y_j$, then $b_i=g(y_i)=g(y_j)=b_j$, contradicting the matching property in \eqref{eq:close-matching-X}.

	For $i\ne j$, define
	\[
		d_{ij}
		=
		|u_{y_j,b_i}|\,|u_{y_j,a_i}|,
		\qquad d_{ii}=0.
	\]
	Since $(y_j)_j$, $(a_i)_i$, and $(b_i)_i$ are injective, the Cauchy--Schwarz inequality gives
    \[
        \sum_jd_{ij}
        \le
        \left(\sum_j|u_{y_j,b_i}|^2\right)^{1/2}
        \left(\sum_j|u_{y_j,a_i}|^2\right)^{1/2}
        \le1
    \]
    for every $i$, and
    \[
        \sum_id_{ij}
        \le
        \left(\sum_i|u_{y_j,b_i}|^2\right)^{1/2}
        \left(\sum_i|u_{y_j,a_i}|^2\right)^{1/2}
        \le1
    \]
    for every $j$.
	Use Lemma~\ref{lem:double-paving} with $\eta=\alpha\beta/2$ to obtain a finite partition $I=I_1\sqcup\cdots\sqcup I_r$. For each $k$, set
    \[
        Y_k=\{y_j:j\in I_k\},
        \qquad
        P_k=\chi_{Y_k}\in\ell^\infty(Y).
    \]
    The coefficient sum below is absolutely convergent by the Cauchy--Schwarz inequality. For $i\in I_k$, equations \eqref{eq:selected-coefficients} and \eqref{eq:yi-close-X} give
	\begin{align}
	\left|
	\langle U^*P_kU\delta_{b_i},\delta_{a_i}\rangle
	\right|
	&=
	\left|
	\sum_{j\in I_k}u_{y_j,b_i}\overline{u_{y_j,a_i}}
	\right|\notag\\
	&\ge
	|u_{y_i,b_i}|\,|u_{y_i,a_i}|
	-
	\sum_{\substack{j\in I_k\\j\ne i}}d_{ij}\notag\\
	&>\alpha\beta-\frac{\alpha\beta}{2}
	=\frac{\alpha\beta}{2}.
	\label{eq:close-threshold-X}
	\end{align}
	By \eqref{eq:diagonal-inclusion}, $P_k\in\Cu(Y,\cF)$, and hence
	\[
		U^*P_kU=\Phi^{-1}(P_k)\in\Cu(X,\cE).
	\]
	Lemma~\ref{lem:threshold} and \eqref{eq:close-threshold-X} show that every paving class in \eqref{eq:close-matching-X} is an entourage. Finite union closure gives $Q_X\in\cE$. Since $\cE$ is closed under inverses,
	\[
		\{(g(f(x)),x):x\in X\}=Q_X^{-1}\in\cE.
	\]
	Thus $g\circ f$ is close to $\id_X$.
\end{proof}

The same proof, with the roles of $X$ and $Y$ reversed, completes the symmetric half. For clarity, we spell out the coefficient that plays the role of \eqref{eq:close-threshold-X}.

\begin{prop}\label{prop:close-Y}
	The map $f\circ g$ is close to $\id_Y$.
\end{prop}

\begin{proof}
	The relation
	\[
		Q_Y=\{(y,f(g(y))):y\in Y\}
	\]
	has horizontal sections of cardinality one. Its vertical sections have cardinality at most $\alpha^{-2}\beta^{-2}$, by the same fiber-counting estimate used for $Q_X$. It is therefore a finite union of matchings. Fix one matching $\{(c_i,d_i):i\in I\}$ and put $x_i=g(c_i)$, so that $d_i=f(x_i)$. The family $(x_i)_i$ is injective, because equality $x_i=x_j$ would imply $d_i=f(x_i)=f(x_j)=d_j$.

	Define
    \[
        d_{ij}'=
        \begin{cases}
            |u_{c_i,x_j}|\,|u_{d_i,x_j}|,&i\ne j,\\
            0,&i=j.
        \end{cases}
    \]
    For every $i$, injectivity of $(x_j)_j$ and the Cauchy--Schwarz inequality give
    \[
        \sum_jd_{ij}'
        \le
        \left(\sum_j|u_{c_i,x_j}|^2\right)^{1/2}
        \left(\sum_j|u_{d_i,x_j}|^2\right)^{1/2}
        \le1.
    \]
    For every $j$, injectivity of $(c_i)_i$ and $(d_i)_i$ gives $\sum_i d_{ij}'\le1$. Apply Lemma~\ref{lem:double-paving} with $\eta=\alpha\beta/2$. For a resulting paving class $I_k$, set
	\[
        X_k=\{x_j:j\in I_k\},
        \qquad
		R_k=\chi_{X_k}\in\ell^\infty(X).
	\]
    The coefficient sum below is absolutely convergent by the Cauchy--Schwarz inequality. Then, for $i\in I_k$,
	\[
	\begin{aligned}
		|\langle UR_kU^*\delta_{d_i},\delta_{c_i}\rangle|
		&\ge
		|u_{c_i,x_i}|\,|u_{d_i,x_i}|
		-
		\sum_{\substack{j\in I_k\\j\ne i}}d_{ij}'\\
		&>\alpha\beta-\frac{\alpha\beta}{2}
		=\frac{\alpha\beta}{2}.
	\end{aligned}
	\]
	Since $UR_kU^*=\Phi(R_k)\in\Cu(Y,\cF)$, Lemma~\ref{lem:threshold} shows that every paving class is an entourage. Therefore $Q_Y\in\cF$, and inverse closure gives
	\[
		\{(f(g(y)),y):y\in Y\}\in\cF.
	\]
	Hence $f\circ g$ is close to $\id_Y$.
\end{proof}

The preceding section gives a coarse equivalence. We now show that the particular coarse equivalence obtained from the implementing unitary can be replaced, within its closeness class, by a bijection.

\section{Upgrading to a bijective coarse equivalence}\label{sec:bijective}

We first record the operator-norm paving theorem in the form needed below. The finite-dimensional self-adjoint case is the paving theorem obtained
from the solution of the celebrated Kadison--Singer problem by Marcus, Spielman,
and Srivastava; see \cite[Theorem~6.1]{MSS15}. The extensions to
arbitrary index sets and to non-self-adjoint operators follow from a
compactness argument and a common refinement applied to the real and
imaginary parts, respectively.

\begin{lem}\label{lem:operator-paving}
	For every $\varepsilon>0$, there is an integer $r\ge1$ with the following property. If $I$ is an arbitrary set and $T\in B(\ell^2(I))$ has zero diagonal, then there is a partition
$
		I=I_1\sqcup\cdots\sqcup I_r
$
	such that
	\[
		\|p_{I_j}Tp_{I_j}\|
		\le \varepsilon\|T\|,
		\qquad 1\le j\le r,
	\]
	where $p_A=\chi_A$ denotes the diagonal projection associated with $A\subseteq I$.
\end{lem}

\begin{proof}
	We first prove the assertion for self-adjoint operators, allowing an
	arbitrary tolerance $\delta>0$. By the finite-dimensional paving
	theorem obtained from the solution of the Kadison--Singer problem
	\cite[Theorem~6.1]{MSS15}, there is an integer
	$r_0=r_0(\delta)$, independent of the dimension, such that every
	finite-dimensional zero-diagonal self-adjoint operator $S$ admits a
	partition into $r_0$ diagonal blocks, each having norm at most
	$\delta\|S\|$.
	
	Let $T=T^*$, and let $I$ be arbitrary. Consider the compact product
	space
$
	\Omega=\{1,\ldots,r_0\}^{I}.
$
	An element $c\in\Omega$ will be regarded as a coloring of $I$. For
	each finite set $K\subseteq I$ and each
	$j\in\{1,\ldots,r_0\}$, let
	\[
	\mathcal C_{K,j}
	=
	\left\{
	c\in\Omega:
	\left\|
	p_{K\cap c^{-1}(j)}T
	p_{K\cap c^{-1}(j)}
	\right\|
	\le\delta\|T\|
	\right\}.
	\]
	Each $\mathcal C_{K,j}$ is closed, since its defining condition
	depends on only finitely many coordinates of $c$.
	
	We claim that the family
	\[
	\{\mathcal C_{K,j}:
	K\subseteq I\text{ finite},\ 1\le j\le r_0\}
	\]
	has the finite intersection property. Indeed, consider finitely many
	of these conditions, involving finite sets
	$K_1,\ldots,K_s$, and put
$
	K=K_1\cup\cdots\cup K_s.
$
	The compression $p_KTp_K$ is a finite-dimensional zero-diagonal
	self-adjoint operator. By the finite-dimensional paving theorem,
	there is a coloring
	\[
	c_K:K\longrightarrow\{1,\ldots,r_0\}
	\]
	such that
	\[
	\left\|
	p_{K\cap c_K^{-1}(j)}T
	p_{K\cap c_K^{-1}(j)}
	\right\|
	\le
	\delta\|p_KTp_K\|
	\le
	\delta\|T\|
	\]
	for every $j$. For each $\nu$, the corresponding block over
	$K_\nu$ is a compression of the block over $K$, and therefore
	satisfies the same norm bound. Extending $c_K$ arbitrarily to $I$
	gives a coloring satisfying all the prescribed conditions.
	
	Compactness of $\Omega$ now yields a coloring
	$c\in\Omega$ belonging to every $\mathcal C_{K,j}$. Set
	\[
	I_j=c^{-1}(j),
	\qquad 1\le j\le r_0,
	\]
	and let
	\[
	S_j=p_{I_j}Tp_{I_j}.
	\]
	If $\xi\in\ell^2(I)$ is a finite-support unit vector and
	$K=\supp(\xi)$, then, with $\eta=p_{I_j}\xi$, we have
	\[
	\begin{aligned}
		|\langle S_j\xi,\xi\rangle|
		&=
		|\langle T\eta,\eta\rangle|\\
		&\le
		\left\|
		p_{K\cap I_j}T p_{K\cap I_j}
		\right\|\|\eta\|^2\\
		&\le
		\delta\|T\|.
	\end{aligned}
	\]
	Finite-support vectors are dense in $\ell^2(I)$, and $S_j$ is
	self-adjoint. The variational formula for the norm therefore gives
	\[
	\|S_j\|\le\delta\|T\|,
	\qquad 1\le j\le r_0.
	\]
	This proves the assertion for self-adjoint operators on arbitrary
	index sets.
	
	We now return to a general zero-diagonal operator $T$. Apply the
	self-adjoint case with tolerance $\delta=\varepsilon/2$ separately
	to
	\[
	\Re T=\frac{T+T^*}{2}
	\qquad\text{and}\qquad
	\Im T=\frac{T-T^*}{2i}.
	\]
	Both operators are self-adjoint and have zero diagonal. Let
	\[
	I=A_1\sqcup\cdots\sqcup A_{r_0},
	\qquad
	I=B_1\sqcup\cdots\sqcup B_{r_0}
	\]
	be the resulting partitions. Their common refinement consists of
	the sets
	\[
	A_k\cap B_\ell,
	\qquad 1\le k,\ell\le r_0,
	\]
	and hence has at most $r_0^2$ classes. For every class $C$ in this
	refinement,
	\[
	\begin{aligned}
		\|p_CTp_C\|
		&\le
		\|p_C(\Re T)p_C\|
		+
		\|p_C(\Im T)p_C\|\\
		&\le
		\frac{\varepsilon}{2}\|\Re T\|
		+
		\frac{\varepsilon}{2}\|\Im T\|\\
		&\le
		\varepsilon\|T\|.
	\end{aligned}
	\]
	Thus the conclusion holds with
	$r=r_0(\varepsilon/2)^2$.
\end{proof}

The second auxiliary lemma converts an invertible, locally finite matrix into a perfect matching.

\begin{lem}\label{lem:invertible-support-matching}
	Let $I$ and $J$ be arbitrary sets, and let
$
		A:\ell^2(I)\longrightarrow\ell^2(J)
$
	be invertible. Suppose that $R=\supp(A)\subseteq J\times I$ has uniformly finite horizontal and vertical sections. Then there is a bijection $\sigma:I\to J$ such that
	\[
		(\sigma(i),i)\in R,
		\qquad i\in I.
	\]
\end{lem}

\begin{proof}
	For a finite set $K\subseteq I$, the vectors $\{A\delta_i:i\in K\}$ are linearly independent and belong to
	\[
		\ell^2(R[K]),
		\qquad
		R[K]=\{j\in J:(j,i)\in R\text{ for some }i\in K\}.
	\]
	Consequently,
	\[
		|R[K]|\ge |K|.
	\]
	Applying the same argument to $A^*$ gives
	\[
		|R^{-1}[L]|\ge |L|
	\]
	for every finite $L\subseteq J$. To construct the first injection, consider the compact product
	\[
		\prod_{i\in I}\{j\in J:(j,i)\in R\},
	\]
whose factors are nonempty by injectivity of $A$ and finite by the local
finiteness of $R$. Regarding a point of this product as a map $\varphi$ on $I$, impose, for each pair of distinct points $i,i'\in I$, the closed condition $\varphi(i)\ne\varphi(i')$. Every finite family of these conditions is satisfiable by the finite Hall marriage theorem, so the finite intersection property gives an injection
	\[
		\varphi:I\longrightarrow J
	\]
	such that $(\varphi(i),i)\in R$ for all $i\in I$. Applying the same argument to $R^{-1}$ gives an injection
	\[
		\psi:J\longrightarrow I
	\]
	such that $(j,\psi(j))\in R$ for all $j\in J$.

	For completeness, the matching version of the Cantor--Bernstein construction preserves this support condition. Set
	\[
		I_0=I\setminus\psi(J),
		\qquad
		I_{n+1}=\psi(\varphi(I_n)),
		\qquad
		I_\infty=\bigcup_{n\ge0}I_n,
	\]
	and define
	\[
		\sigma(i)=
		\begin{cases}
			\varphi(i),&i\in I_\infty,\\
			\psi^{-1}(i),&i\notin I_\infty.
		\end{cases}
	\]
	Then $\sigma$ is a bijection, and $(\sigma(i),i)\in R$ in both cases.
\end{proof}

We now apply these two lemmas to the controlled coefficient maps constructed above.

\begin{prop}\label{prop:bijective-upgrade}
	There is a bijection $b:X\to Y$ such that $b\sim f$ and $b^{-1}\sim g$. In particular, $b$ and $b^{-1}$ are controlled.
\end{prop}

\begin{proof}
	By \eqref{eq:fibre-bounds}, the fibers of $f$ have uniformly bounded cardinality. Set
	\[
		K_f=\sup_{y\in Y}|f^{-1}(y)|<\infty.
	\]
	Color each fiber with at most $K_f$ colors. The color classes give a finite partition of $X$ into sets on each of which $f$ is injective. Fix one such set $A$. Let
	\[
		J_A:\ell^2(A)\longrightarrow\ell^2(f(A)),
		\qquad
		J_A\delta_x=\delta_{f(x)},
	\]
	and consider
	\[
		T_A=J_A^*p_{f(A)}Up_A\in B(\ell^2(A)).
	\]
	Write $T_A=D_A+C_A$, where $D_A$ is diagonal and $C_A$ has zero diagonal. By \eqref{eq:selected-coefficients},
	\[
		|(D_A)_{xx}|=|u_{f(x),x}|\ge\alpha,
		\qquad x\in A,
	\]
	so $D_A$ is invertible and $\|D_A^{-1}\|\le\alpha^{-1}$. Since $T_A$ is a compression of the unitary $U$, we have $\|T_A\|\le1$; also $\|D_A\|\le1$. Hence $\|C_A\|\le2$. Applying Lemma~\ref{lem:operator-paving} with $\varepsilon=\alpha/8$ and refining each of the finitely many injectivity classes, we obtain a finite partition
	\begin{equation}\label{eq:bijective-partition}
		X=B_1\sqcup\cdots\sqcup B_m
	\end{equation}
	such that $f|_{B_i}$ is injective and
	\begin{equation}\label{eq:invertible-compression}
		p_{f(B_i)}Up_{B_i}:
		p_{B_i}\ell^2(X)\longrightarrow
		p_{f(B_i)}\ell^2(Y)
	\end{equation}
	is invertible. Indeed, in the coordinates given by $J_A$, the operator in \eqref{eq:invertible-compression} is $D_{B_i}+C_{B_i}$, where
	\[
		\|C_{B_i}\|\le\frac{\alpha}{4},
		\qquad
		\|D_{B_i}^{-1}C_{B_i}\|\le\frac14<1.
	\]
	The factorization
	\[
		D_{B_i}+C_{B_i}
		=
		D_{B_i}\bigl(1+D_{B_i}^{-1}C_{B_i}\bigr)
	\]
	and the Neumann series show that the operator is invertible. Moreover, for every unit vector $\xi$,
	\[
		\|(D_{B_i}+C_{B_i})\xi\|
		\ge
		\|D_{B_i}\xi\|-\|C_{B_i}\xi\|
		\ge
		\alpha-\frac{\alpha}{4}
		\ge
		\frac{\alpha}{2}.
	\]
	Thus the minimum modulus of every operator in \eqref{eq:invertible-compression} is at least $\alpha/2$.

	Put
	\[
		P_i=Up_{B_i}U^*=\Phi(p_{B_i}),
		\qquad
		Q_i=p_{f(B_i)}.
	\]
	Then $P_i,Q_i\in\Cu(Y,\cF)$, and \eqref{eq:invertible-compression} says that
	\[
		Q_iP_i:P_i\ell^2(Y)\longrightarrow Q_i\ell^2(Y)
	\]
	is invertible. Hence
	\[
		P_iQ_iP_i\ge\frac{\alpha^2}{4}P_i.
	\]
	Functional calculus applied to $P_iQ_iP_i$ yields the polar partial isometry of $Q_iP_i$. Since $Q_iP_i:P_i\ell^2(Y)\to Q_i\ell^2(Y)$ is onto, the final projection of this polar part is exactly $Q_i$. Thus there is a partial isometry $v_i\in\Cu(Y,\cF)$ satisfying
	\begin{equation}\label{eq:polar-equivalence}
		v_i^*v_i=P_i,
		\qquad
		v_iv_i^*=Q_i.
	\end{equation}
	Explicitly, one may take
	\[
		v_i=Q_iP_i h(P_iQ_iP_i),
	\]
	where $h\in C([0,1])$ is chosen so that $h(0)=0$ and $h(t)=t^{-1/2}$ for $t\in[\alpha^2/4,1]$.

	Define
	\[
		V=
		\begin{pmatrix}
			v_1\\ \vdots\\ v_m
		\end{pmatrix}
		\in M_{m,1}(\Cu(Y,\cF)),
		\qquad
		Q=\diag(Q_1,\ldots,Q_m).
	\]
	Since the $P_i$ are pairwise orthogonal and sum to the identity, and since $v_iv_j^*=v_iP_iP_jv_j^*=0$ for $i\ne j$, \eqref{eq:polar-equivalence} gives
	\begin{equation}\label{eq:matrix-equivalence}
		V^*V=1,
		\qquad
		VV^*=Q.
	\end{equation}
	Thus $V$ is a unitary from $\ell^2(Y)$ onto $Q\ell^2(Y)^m$.

Since $m<\infty$ and entourage-supported operators are norm dense in
$\Cu(Y,\cF)$, we may approximate the entries of $V$ simultaneously.
For each $i\in\{1,\ldots,m\}$, choose an entourage-supported operator
$w_i\in\Cu(Y,\cF)$ such that, for the column
$
W=(w_1,\ldots,w_m)^{\mathsf T},
$
one has
\[
\|W-V\|<\frac12.
\]
For each $i$, choose $F_i\in\cF$ such that
$
\supp(w_i)\subseteq F_i.
$
Set
$
a=QW.
$
	Since $QV=V$, we have $\|a-V\|<1/2$. Moreover, by \eqref{eq:matrix-equivalence},
	\[
		a=VV^*a=V(V^*a),
		\qquad
		\|V^*a-1\|<\frac12.
	\]
	It follows that
	\begin{equation}\label{eq:transport-isomorphism}
		a:\ell^2(Y)\longrightarrow Q\ell^2(Y)^m
	\end{equation}
	is invertible.

	Identify the range space in \eqref{eq:transport-isomorphism} with $\ell^2(S)$, where
	\[
		S=\bigsqcup_{i=1}^m\bigl(f(B_i)\times\{i\}\bigr).
	\]
	The relation $R=\supp(a)\subseteq S\times Y$ has uniformly finite sections in both directions. Indeed, if $((z,i),y)\in R$, then $(z,y)\in F_i$. Consequently,
	\[
		\sup_{(z,i)\in S}|\{y:((z,i),y)\in R\}|
		\le \max_{1\le i\le m}N_{F_i},
		\qquad
		\sup_{y\in Y}|\{(z,i):((z,i),y)\in R\}|
		\le \sum_{i=1}^m N_{F_i^{-1}}.
	\]
	Lemma~\ref{lem:invertible-support-matching} therefore provides a bijection
	\[
		\sigma:Y\longrightarrow S
	\]
	such that $(\sigma(y),y)\in R$ for every $y\in Y$.

	On the other hand, \eqref{eq:bijective-partition} and the injectivity of each $f|_{B_i}$ show that
	\[
		\theta:X\longrightarrow S,
		\qquad
		\theta(x)=(f(x),i)\quad(x\in B_i),
	\]
	is a bijection. Define
	\[
		b=\sigma^{-1}\circ\theta:X\longrightarrow Y.
	\]
	If $x\in B_i$, then $\sigma(b(x))=(f(x),i)$. Since this is an edge of $R$, we have
	\[
		(f(x),b(x))\in F_i.
	\]
	As $m$ is finite,
	\[
		\{(f(x),b(x)):x\in X\}
		\subseteq F_1\cup\cdots\cup F_m\in\cF.
	\]
	Thus $f\sim b$, and hence $b\sim f$ by symmetry of closeness. To spell out the standard closeness argument, let
	\[
		D=\{(b(x),f(x)):x\in X\}\in\cF.
	\]
	For every $E\in\cE$, the relation $(b\times b)(E)$ is contained in the composition of $D$, $(f\times f)(E)$, and $D^{-1}$, and is therefore an entourage. Hence $b$ is controlled. Since $g$ is controlled and $g\circ f\sim\id_X$, we also have
	\[
		g\circ b\sim g\circ f\sim\id_X.
	\]
	Substituting $x=b^{-1}(y)$ shows that $g\sim b^{-1}$, and hence $b^{-1}\sim g$ by symmetry of closeness. Again, closeness to the controlled map $g$ implies that $b^{-1}$ is controlled.
\end{proof}

We can now finish the proof of the main result.

\begin{proof}[Proof of Theorem~\ref{thm:main}]
	If one of $X$ and $Y$ is empty, then so is the other, and the conclusion is immediate. We therefore assume that both spaces are nonempty. Proposition~\ref{prop:spatial} spatially implements the given isomorphism. Corollary~\ref{cor:two-sided-coefficients} and \eqref{eq:f-g-selection} provide maps $f:X\to Y$ and $g:Y\to X$ satisfying the uniform coefficient and fiber bounds in \eqref{eq:selected-coefficients} and \eqref{eq:fibre-bounds}. Proposition~\ref{prop:f-controlled} and Corollary~\ref{cor:g-controlled} show that both maps are controlled. Finally, Propositions~\ref{prop:close-X} and~\ref{prop:close-Y} show that
	\[
		g\circ f\sim\id_X,
		\qquad
		f\circ g\sim\id_Y.
	\]
	Proposition~\ref{prop:bijective-upgrade} replaces $f$ by a close bijection $b:X\to Y$ such that $b^{-1}\sim g$. Hence both $b$ and $b^{-1}$ are controlled, and $(X,\cE)$ and $(Y,\cF)$ are bijectively coarsely equivalent.
\end{proof}

\begin{rem}
	Noncompact ghost projections obstruct naive attempts to localize each rank-one image $\Phi(e_{xx})$ separately. Lemma~\ref{lem:random-diagonal} avoids this obstruction by using the entire atomic diagonal algebra at once: all random diagonal unitaries are transported into the target uniform Roe algebra. This distinction is why the argument does not require the compact-ghost hypotheses used in \cite[Theorem~1.4]{BFV22}. For the equivalence between property~A and compactness of all ghost operators, see Roe and Willett \cite[Theorem~1.3]{RW14}.
\end{rem}

\end{document}